\documentclass[10pt]{article}
\usepackage{amssymb}
\usepackage{amsthm}
\usepackage{amsmath}
\usepackage{stmaryrd}
\usepackage{array}
\usepackage[titletoc,toc,title]{appendix}

\usepackage{enumitem}

\usepackage{pdflscape}

\usepackage{titlesec}
\titlelabel{\thetitle.\quad}

\usepackage{authblk}

\newtheorem{definition}{Definition}[section]
\newtheorem{theorem}[definition]{Theorem}

\newtheorem{lemma}[definition]{Lemma}

\newtheorem{remark}[definition]{Remark}
\newtheorem{corollary}[definition]{Corollary}

\newtheorem{proposition}[definition]{Proposition}

\def\N{{\mathbb N}}
\def\F{{\mathbb F}}
\def\Z{{\mathbb Z}}

\def\PSL2{PSL_2 (\mathbb{Z})}
\def\SL2{SL_2 (\mathbb{Z})}

\def\lbrack{\left[}
\def\rbrack{\right]}

\def\ad{{\rm ad}\ }

\def\indset{\mathcal{I}}

\def\indsetL4{\indset^*}

\def\Basis5{\mathcal{B}}
\def\Reps5{\mathcal{S}}

\def\freeassoc3{\F\left< A,B,C\right>}

\def\algI{I}

\def\genIarb{I}
\def\genI{\genIarb_1}
\def\genII{\genIarb_2}
\def\genIII{\genIarb_3}
\def\Casimir{C}
\def\FairO{U_q'(\mathfrak{so}_3)}
\def\FairOLie{\mathfrak{L}}
\def\TorAlg{\mathcal{A}_q}
\def\TorLie{\mathcal{L}_q}

\def\genzarb{z}
\def\genz{\genzarb_1}
\def\genzz{\genzarb_2}
\def\genzzz{\genzarb_3}

\def\igenzarb{\genzarb^{-1}}
\def\igenz{\genzarb_1^{-1}}
\def\igenzz{\genzarb_2^{-1}}
\def\igenzzz{\genzarb_3^{-1}}

\def\pmgenz{z_1^{\pm 1}}
\def\pmgenzz{z_2^{\pm 1}}
\def\pmgenzzz{z_3^{\pm 1}}

\def\qhalf{q^{\frac{1}{2}}}
\def\nqhalf{q^{-\frac{1}{2}}}

\def\genGarb{G}
\def\genG{\genGarb_1}
\def\genGG{\genGarb_2}
\def\genGGG{\genGarb_3}

\def\Iso{\Phi}
\def\Proj{\pi}

\def\gradsub{\Lambda}
\def\TorGrad{\{\gradsub_N\  :\  N\in\Z\}}

\def\Equit{U_q\left(\mathfrak{sl}_2\right)}

\def\trigradsub{A}



\begin{document}

\title{{\bf A Casimir element inexpressible as a Lie polynomial}}
\author{{\scshape Rafael Reno S. Cantuba} \\
Mathematics and Statistics Department\\
De La Salle University Manila\\
{ Taft Ave., Manila, Philippines}\\
{\it rafael{\_}cantuba@dlsu.edu.ph}}

\date{}

\maketitle

\begin{abstract}

Let $q$ be a scalar that is not a root of unity. We show that any polynomial in the Casimir element of the Fairlie-Odesskii algebra $U_q'(\mathfrak{so}_3)$ cannot be expressed in terms of only Lie algebra operations performed on the generators $I_1,I_2,I_3$ in the usual presentation of $U_q'(\mathfrak{so}_3)$. Hence, the vector space sum of the center of $U_q'(\mathfrak{so}_3)$ and the Lie subalgebra of $U_q'(\mathfrak{so}_3)$ generated by $I_1,I_2,I_3$ is direct.\\


\end{abstract}

\section{Introduction}

Important quantum groups and associative algebras have presentations with deformed commutation relations, by which we mean defining relations that involve expressions of the form $c_1UV-c_2VU$ where $U,V$ are elements of the associative algebra, and the scalar parameters $c_1,c_2$ are not necessarily both $1$. These expressions resemble the usual Lie bracket in the associative algebra, which is $\lbrack U,V\rbrack:=UV-VU$. Some specific examples are the quantum group $\Equit$ in its ``equitable presentation'' \cite{Ito}, the Fairlie-Odesskii algebra $\FairO$ in its usual presentation \cite{Hav99,Hav01,Hav11,Iorg}, the parametric family of Askey-Wilson algebras \cite{Ter,Zhed}, and the parametric family of $q$-deformed Heisenberg algebras \cite{Hel00,Hel05} whose defining relations are $q$-deformations of the Heisenberg-Weyl relation.

However, the existence of an associative algebra structure in some vector space implies the existence of a Lie algebra structure in the same vector space. Thus, even if the defining relations of an associative algebra involve deformed commutation relations, it is still possible to compute Lie polynomials in the generators of the algebra. This notion first appeared in \cite[Problem 12.14]{Ter} for the universal Askey-Wilson algebra. It was found that the defining relations in this algebra and the related family of Askey-Wilson algebras are not Lie polynomials in the generators \cite{Can15}, and so the usual algebraic machinery for finitely generated and finitely presented Lie algebras is not directly applicable, but it was further shown that the Lie subalgebra generated by the generators of the algebra is not free \cite{Can15}. Similar studies were done for $q$-deformed Heisenberg algebras \cite{Can17,Can18}. In all such studies \cite{Can15,Can17,Can18} the focus was on the consequences of the non-Lie polynomial, deformed commutation relations on the Lie polynomials in the same algebra.

For this paper, we study the Fairlie-Odesskii algebra $\FairO$, and the consequences of its deformed commutation relation on the Lie polynomials in relation to the Casimir element, an element which is significant in studying representations and central elements of the algebra \cite{Hav99,Hav01,Hav11}. More specifically, we show that any nonzero polynomial in the Casimir element of the Fairlie-Odesskii algebra is not a Lie polynomial in the generators of the algebra.

\section{Preliminaries}
Denote the set of all nonnegative integers by $\N$, and the set of all positive integers by $\Z^+$. Let $\F$ denote a fixed but arbitrary field. Throughout, by an \emph{algebra} we mean a unital associative algebra $\mathcal{A}$ over $\F$ with unity element $\algI_\mathcal{A}$. We use the convention that $U^0=\algI_\mathcal{A}$ for any $U\in\mathcal{A}$, and we denote $\algI_\mathcal{A}$ simply by $\algI$ if no confusion shall arise. Any subalgebra is assumed to contain the unity element. We also note that every algebra $\mathcal{A}$ has a Lie algebra structure induced by $\lbrack U,V\rbrack:=UV-VU$ for all $U,V\in\mathcal{A}$. Throughout, whenever we refer to a Lie algebra structure on an algebra $\mathcal{A}$, we shall always mean that which is induced by the Lie bracket operation $\lbrack\cdot,\cdot\rbrack$ just mentioned. Let $X_1,X_2,\ldots,X_n\in\mathcal{A}$. If $\mathcal{K}$ is the Lie subalgebra of $\mathcal{A}$ generated by $X_1,X_2,\ldots,X_n$, then we call the elements of $\mathcal{K}$ the \emph{Lie polynomials in $X_1,X_2,\ldots,X_n$}. Given $U\in\mathcal{A}$, the linear map $\ad U:\mathcal{A}\rightarrow\mathcal{A}$ is defined by the rule $V\mapsto\lbrack U,V\rbrack$. Also, for any linear map $\varphi$ whose domain and codomain are equal, and for any $n\in\N$, by $\varphi^n$ we mean composition of $\varphi$ with itself $n$ times, where $\varphi^0$ is interpreted as the identity linear map.\\

\noindent {\bf The Fairlie-Odeskii algebra.} Fix a nonzero $q\in\F$. The \emph{Fairlie-Odeskii algebra} is the algebra $\FairO$ that has a presentation by generators $\genI$, $\genII$, $\genIII$ and relations 
\begin{eqnarray}
\qhalf\genI\genII - \nqhalf\genII\genI = \genIII,\quad\quad\quad\quad \qhalf\genII\genIII - \nqhalf\genIII\genII = \genI,\quad\quad\quad\quad \qhalf\genIII\genI - \nqhalf\genI\genIII = \genII.\nonumber
\end{eqnarray}
By Bergman's Diamond Lemma \cite[Theorem 1.2]{Berg}, a basis for $\FairO$ consists of the vectors $\genI^h\genII^m\genIII^n$ for all $h,m,n\in\N$. The element
\begin{eqnarray}
\Casimir:=-\qhalf(q-q^{-1})\genI\genII\genIII + q\genI^2 + q^{-1}\genII^2+q\genIII^2\label{CasimirDef}
\end{eqnarray}
of $\FairO$ is called the \emph{Casimir element} of $\FairO$. If $q$ is not a root of unity, it is known that $\Casimir$ generates the center of $\FairO$ \cite[Theorem II]{Hav11}. Denote by $\FairOLie$ the Lie subalgebra of $\FairO$ generated by $\genI,\genII,\genIII$.\\

\noindent {\bf A torus algebra related to $\FairO$.} The algebra $\FairO$ can be interpreted as an algebra of quantum geodesics or an algebra of quantized geodesic functions on a coordinate algebra of a torus, which is related to quantum gravity \cite{Chek,Fair,Iorg,Odess}. More precisely, $\FairO$ is a subalgebra of some other algebra related to a torus, which we describe in the following. Denote by $\TorAlg$ the algebra with a presentation given by six generators $\pmgenz$, $\pmgenzz$, $\pmgenzzz$ that satisfy the relations
\begin{eqnarray}
 \genz\genzz=q\genzz\genz,\quad\quad\quad\quad \genzz\genzzz=q\genzzz\genzz,\quad\quad\quad\quad \genzzz\genz=q\genz\genzzz,\quad\quad\quad\quad\genzarb_k\igenzarb_k=\algI=\igenzarb_k\genzarb_k,\label{TorDefRels}
\end{eqnarray}
for all $k\in\{1,2,3\}$. An immediate consequence of $\TorAlg$ having the above presentation is that there exists an algebra isomorphism $\Iso:\TorAlg\rightarrow\TorAlg$ that performs the assignments
\begin{eqnarray}
\Iso & : &\genz\mapsto\genzz\mapsto\genzzz\mapsto\genz,\nonumber\\
& &\igenz\mapsto\igenzz\mapsto\igenzzz\mapsto\igenz.\nonumber
\end{eqnarray}
Also by \cite[Theorem 1.2]{Berg}, a basis for $\TorAlg$ consists of 
\begin{eqnarray}
\genzzz^h\genzz^m\genz^n,\quad\quad\quad\quad\quad (h,m,n\in\Z.)\label{TorBasis}
\end{eqnarray}
The elements 
\begin{eqnarray}
\genG & := & \nqhalf\igenzzz\igenz+\qhalf\igenzzz\genz+\nqhalf\genzzz\genz,\nonumber\\
\genGG & := & \nqhalf\igenzz\igenzzz+\qhalf\igenzz\genzzz+\nqhalf\genzz\genzzz,\nonumber\\
\genGGG & := & \nqhalf\igenz\igenzz+\qhalf\igenz\genzz+\nqhalf\genz\genzz,\nonumber
\end{eqnarray}
of $\TorAlg$ are of significance as shall be evident in the results that follow. Initially, we have the following.
\begin{proposition}[{\cite[Proposition 3.1]{Iorg}}]\label{subalgProp} There exists an injective algebra homomorphism $\FairO\rightarrow\TorAlg$ such that $$\genIarb_k\mapsto\frac{\genGarb_k}{q-q^{-1}}$$ for all $k\in\{1,2,3\}$.
\end{proposition}
\noindent Thus, we identify $\FairO$ as the subalgebra of $\TorAlg$ as described by Proposition~\ref{subalgProp} above.

\section{Reordering formula in $\TorAlg$}\label{ReorderSec}

Since the fact that the vectors \eqref{TorBasis} form a basis for $\TorAlg$ can be shown using the Diamond Lemma, this means that in a finite number of steps, any element of $\TorAlg$ can be written uniquely as a linear combination of \eqref{TorBasis}. As an example, by simple use of the relations \eqref{TorDefRels}, the elements $\genG,\genGG,\genGGG$ can be rewritten as
\begin{eqnarray}
\genG & = & \nqhalf\igenzzz\igenz+\qhalf\igenzzz\genz+\nqhalf\genzzz\genz,\label{normalG1}\\
\genGG & = & \qhalf\igenzzz\igenzz+\nqhalf\igenzzz\genzz+\qhalf\genzzz\genzz,\label{normalG2}\\
\genGGG & = & \qhalf\igenzz\igenz+\nqhalf\igenzz\genz+\qhalf\genzz\genz,\label{normalG3}
\end{eqnarray}

\noindent in terms of the basis \eqref{TorBasis} of $\TorAlg$. In the theorems that follow, we show explicitly some reordering formula for $\TorAlg$.

\begin{proposition}\label{reorderProp} The relations
\begin{eqnarray}
\genz^m\genzz^n & = & q^{mn}\genzz^n\genz^m,\label{comm12}\\
\genzz^m\genzzz^n & = & q^{mn}\genzzz^n\genzz^m,\label{comm23}\\
\genz^m\genzzz^n & = & q^{-mn}\genzzz^n\genz^m,\label{comm13}
\end{eqnarray}
hold in $\TorAlg$ for any $m,n\in\Z$.
\end{proposition}
\begin{proof} We first prove \eqref{comm12}. For the case $m,n\in\Z^+$, using the first relation in \eqref{TorDefRels} and double induction on $m,n$ it is routine to show that 
\begin{eqnarray}
\genz^m\genzz^n = q^{mn}\genzz^n\genz^m,\quad (m,n\in\Z^+.)\label{posmposn}
\end{eqnarray}
Multiply the left-hand side of \eqref{posmposn} by $\genz^{-m}$ both from the left and the right. Do similarly for the right-hand side of \eqref{posmposn}, and then solve for $\genz^{-m}\genzz^n$. The result is 
\begin{eqnarray}
\genz^{-m}\genzz^n = q^{-mn}\genzz^n\genz^{-m},\quad (m,n\in\Z^+.)\label{negmposn}
\end{eqnarray}
Similarly, multiply the left-hand side of \eqref{negmposn} by $\genzz^{-n}$ both from the left and the right, and do similarly for the right-hand side of \eqref{negmposn}, and then solve for $\genz^{-m}\genzz^{-n}$. This results to
\begin{eqnarray}
\genz^{-m}\genzz^{-n} = q^{mn}\genzz^{-n}\genz^{-m},\quad (m,n\in\Z^+.)\label{negmnegn}
\end{eqnarray}
Thus, \eqref{comm12} holds for the cases covered in \eqref{posmposn}, \eqref{negmposn}, \eqref{negmnegn}. We now consider arbitrary $m,n\in\Z$. If one of $m,n$ is zero, then we are left with the trivial equation $u=q^0 u$ where $u$ is either $\algI$ or a power of one generator. In this case, \eqref{comm12} still holds. Consider the case that both $m,n$ are nonzero. Given $h\in\{m,n\}$ such that $h$ is negative, there exists $k\in\Z^+$ such that $h=-k$. This implies that one of the three cases \eqref{posmposn}, \eqref{negmposn}, \eqref{negmnegn} is applicable, and so \eqref{comm12} holds for any $m,n\in\Z$. Apply $\Iso$ to both sides of \eqref{comm12} to get \eqref{comm23}. To get \eqref{comm13}, observe that since $m,n$ are arbitrary, \eqref{comm23} can be written as $\genzz^n\genzzz^m = q^{mn}\genzzz^m\genzz^n$. Apply $\Iso$ to both sides of this equation, and then solve for $\genz^m\genzzz^n$. From this, we get \eqref{comm13}.\qed
\end{proof}

\noindent The reordering formula from Proposition~\ref{reorderProp} may be easily used to rewrite the product or Lie bracket of any two basis elements from \eqref{TorBasis} into a linear combination of \eqref{TorBasis}.

\begin{corollary}\label{structCor} Given any integers $h,m,n,u,v,w$, the relations
\begin{eqnarray}
\genzzz^h\genzz^m\genz^n\cdot\genzzz^u\genzz^v\genz^w & = & q^{mu+nv-nu}\genzzz^{h+u}\genzz^{m+v}\genz^{n+w},\label{structassoc}\\
\lbrack\genzzz^h\genzz^m\genz^n,\genzzz^u\genzz^v\genz^w\rbrack & = & \left(q^{mu+nv-nu}-q^{hv-hw+mw}\right)\genzzz^{h+u}\genzz^{m+v}\genz^{n+w},\label{structLie}
\end{eqnarray}
hold in $\TorAlg$.
\end{corollary}
\begin{proof} Reorder $\genzzz^h\genzz^m\genz^n\cdot\genzzz^u\genzz^v\genz^w$ using \eqref{comm12}, \eqref{comm23}, \eqref{comm13}. From this, we get \eqref{structassoc}. The relation \eqref{structLie} is a routine application of \eqref{structassoc}.\qed
\end{proof}

\begin{corollary}\label{assocvvLieCor} Given any integers $h,m,n,u,v,w$, the relation
\begin{eqnarray}
\lbrack\genzzz^h\genzz^m\genz^n,\genzzz^u\genzz^v\genz^w\rbrack = \left(1-q^H\right)\genzzz^h\genzz^m\genz^n\cdot\genzzz^u\genzz^v\genz^w,\label{assocvvLie}
\end{eqnarray}
where $H=h(v-w)+m(w-u)+n(u-v)$, holds in $\TorAlg$.
\end{corollary}
\begin{proof} Use \eqref{structassoc} and \eqref{structLie}.\qed
\end{proof}

\begin{remark}\label{minorRem} With reference to Corollaries~\ref{structCor} and \ref{assocvvLieCor}, we have the following observations.
\begin{enumerate}
\item\label{Z3grad} Denote by $\trigradsub_{(h,m,n)}$ the span of the basis element $\genzzz^h\genzz^m\genz^n$ of $\TorAlg$. The relation \eqref{structassoc} implies that the collection $\{\trigradsub_{(h,m,n)}\  :\  (h,m,n)\in\Z^3\}$ of one-dimensional vector subspaces of $\TorAlg$ is a $\Z^3$-gradation of $\TorAlg$.
\item If $q$ is not a root of unity and if $H\neq 0$, then the identity \eqref{assocvvLie} implies that scalar multiplication can be used to ``convert'' the product of two basis vectors from \eqref{TorBasis} into the Lie bracket of the same two vectors.
\end{enumerate}
\end{remark}

\section{The Lie subalgebra of $\TorAlg$ generated by $\pmgenz$, $\pmgenzz$, $\pmgenzzz$}

Denote by $\TorLie$ the Lie subalgebra of $\TorAlg$ generated by $\pmgenz$, $\pmgenzz$, $\pmgenzzz$. Our initial goal in this section is to identify basis vectors of $\TorAlg$ from \eqref{TorBasis} that are elements of $\TorLie$. From this point onward, assume that $q$ is not a root of unity. As an initial example, observe that by some routine computations involving the use of the reordering formula \eqref{comm12},\eqref{comm23},\eqref{comm13}, we have
\begin{eqnarray}
\genzzz\genzz^2\genz = (1-q)^{-3}\lbrack\genzzz,\lbrack\genzz,\lbrack\genzz,\genz\rbrack\rbrack\rbrack,\label{comm0}
\end{eqnarray}
which proves that $\genzzz\genzz^2\genz\in\TorLie$ since the nonzero scalar $(1-q)^{-3}$ exists in the field $\F$ because of the assumption that $q$ is not a root of unity. We proceed in a rather constructive fashion until we get more basis elements from \eqref{TorBasis} that are also elements of $\TorLie$. Our next step is the following.

\begin{proposition} For any $T\in\Z$, the relation
\begin{eqnarray} \igenzzz\genzz^{-2}\genz^T=\left\{
\begin{array}{ll}
       
      q^T(1-q)^{-T-2}\left(\ad\genz\right)^T\left(\lbrack\lbrack\igenzzz,\igenzz\rbrack,\igenzz\rbrack\right), & T\in\N, \\
     
      (-1)^T(1-q)^{T-2}\left(\ad\igenz\right)^{-T}\left(\lbrack\lbrack\igenzzz,\igenzz\rbrack,\igenzz\rbrack\right), & T\in\Z\backslash\N,
\end{array}\right.\label{commT}
\end{eqnarray}
holds in $\TorAlg$.
\end{proposition}
\begin{proof} We first consider the case $T\in\N$, and we use induction on $T$. By routine calculations using the reordering formula \eqref{comm12}, \eqref{comm23}, \eqref{comm13}, we have
\begin{eqnarray}
\lbrack\lbrack\igenzzz,\igenzz\rbrack,\igenzz\rbrack = (1-q)^2 \igenzzz\genzz^{-2},\nonumber
\end{eqnarray}
which proves that \eqref{commT} holds for $T=0$. Suppose \eqref{commT} holds for some $T\in\N$. That is,
\begin{eqnarray}
\left(\ad\genz\right)^T\left(\lbrack\lbrack\igenzzz,\igenzz\rbrack,\igenzz\rbrack\right)=q^{-T}(1-q)^{T+2}\igenzzz\genzz^{-2}\genz^T.\label{induct1}
\end{eqnarray}
Applying the map $\ad\genz$ on both sides of \eqref{induct1} and using the reordering formula \eqref{comm12}, \eqref{comm23}, \eqref{comm13}, we get
\begin{eqnarray}
\left(\ad\genz\right)^{T+1}\left(\lbrack\lbrack\igenzzz,\igenzz\rbrack,\igenzz\rbrack\right) & = & q^{-T}(1-q)^{T+2}(q^{-1}-1)\igenzzz\genzz^{-2}\genz^{T+1},\nonumber\\
& = & q^{-(T+1)}(1-q)^{(T+1)+2}\igenzzz\genzz^{-2}\genz^{T+1},\nonumber
\end{eqnarray}
which completes the induction for $T\in\N$. For the case $T\in\Z\backslash\N$, we perform induction with decreasing values of $T$: we prove that \eqref{commT} holds for $T=-1$, and prove its validity at $T-1$ given that it holds for some $T\in\Z\backslash\N$. Proceeding as such, by routine calculations involving evaluation of Lie brackets and the reordering formula \eqref{comm12}, \eqref{comm23}, \eqref{comm13}, we find that
\begin{eqnarray}
\left(\ad\igenz\right)\left(\lbrack\lbrack\igenzzz,\igenzz\rbrack,\igenzz\rbrack\right) & = & (1-q)^2(q^{-1}\cdot q^2-1)\igenzzz\genzz^{-2}\igenz,\nonumber
\end{eqnarray}
which implies that 
\begin{eqnarray}
\igenzzz\genzz^{-2}\igenz & = & (-1)(1-q)^{-3}\left(\ad\igenz\right)\left(\lbrack\lbrack\igenzzz,\igenzz\rbrack,\igenzz\rbrack\right),\nonumber
\end{eqnarray}
which proves \eqref{commT} for the case $T=-1$. Suppose that \eqref{commT} holds for some $T\in\Z\backslash\N$. That is,
\begin{eqnarray}
\left(\ad\igenz\right)^{-T}\left(\lbrack\lbrack\igenzzz,\igenzz\rbrack,\igenzz\rbrack\right) = (-1)^{-T}(1-q)^{2-T}\igenzzz\genzz^{-2}\genz^T.\label{induct2}
\end{eqnarray}
We emphasize here that in \eqref{induct2}, the exponent of $\ad\igenz$ is $-T>0$, and so $\left(\ad\igenz\right)^{-T}$  is a valid composition of mappings. We then apply the map $\ad\igenz$ on both sides of \eqref{induct2} and use the reordering formula \eqref{comm12}, \eqref{comm23}, \eqref{comm13}. From this, we get
\begin{eqnarray}
\left(\ad\igenz\right)^{-T+1}\left(\lbrack\lbrack\igenzzz,\igenzz\rbrack,\igenzz\rbrack\right) & = & (-1)^{-T}(1-q)^{2-T}(q^{-1}\cdot q^2-1)\igenzzz\genzz^{-2}\genz^{T-1},\nonumber\\
& = & (-1)^{-T+1}(1-q)^{3-T}\igenzzz\genzz^{-2}\genz^{T-1},\nonumber
\end{eqnarray}
which further implies that
\begin{eqnarray}
\igenzzz\genzz^{-2}\genz^{T-1} = (-1)^{T-1}(1-q)^{(T-1)-2}\left(\ad\igenz\right)^{-(T-1)}\left(\lbrack\lbrack\igenzzz,\igenzz\rbrack,\igenzz\rbrack\right).\nonumber
\end{eqnarray}
By induction \eqref{commT} holds for all $T\in\Z\backslash\N$, and this completes the proof.\qed
\end{proof}
The relevance of \eqref{comm0} and \eqref{commT} will now be apparent in the proof of the following.
\begin{lemma}\label{onepowerLem} For any $h\in\Z\backslash\{0\}$, we have 
\begin{eqnarray}
\genzzz^h,\quad\genzz^h,\quad\genz^h\quad\quad\in\quad\quad\TorLie.\nonumber
\end{eqnarray}
\end{lemma}
\begin{proof} By \eqref{comm0} and \eqref{commT}, both $\genzzz\genzz^2\genz$ and $\igenzzz\genzz^{-2}\genz^{h-1}$ are elements of $\TorLie$. But by \eqref{structLie}, we have
\begin{eqnarray}
\genz^h=q^{3}(1-q^h)^{-1}\lbrack\genzzz\genzz^2\genz,\igenzzz\genzz^{-2}\genz^{h-1}\rbrack,\label{zinL}
\end{eqnarray}
where  the existence of $(1-q^h)^{-1}$ follows from the assumption that  $q$ is not a root of unity. Thus, $\genz^h\in\TorLie$. We now show $\genzz^h\in\TorLie$. Apply $\Iso$ on both sides of \eqref{zinL}, and we obtain
\begin{eqnarray}
\genzz^h = q^3(1-q^h)^{-1}\lbrack\Iso\left(\genzzz\genzz^2\genz\right),\Iso\left(\igenzzz\genzz^{-2}\genz^{h-1}\right)\rbrack,\label{z2inL}
\end{eqnarray}
where the left-hand side was obtained by using the fact that $\Iso$ is an algebra homomorphism, while the right-hand side was obtained using the fact that $\Iso$, being an algebra homomorphism, is necessarily a Lie algebra homomorphism. To evaluate $\Iso\left(\genzzz\genzz^2\genz\right)$ and $\Iso\left(\igenzzz\genzz^{-2}\genz^{h-1}\right)$, we use \eqref{comm0} and \eqref{commT}, and we use the property of $\Iso$ that it is a Lie algebra homomorphism to evaluate the resulting right-hand sides. This gives us
\begin{eqnarray} 
\Iso\left(\genzzz\genzz^2\genz\right) & = & (1-q)^{-3}\lbrack\genz,\lbrack\genzzz,\lbrack\genzzz,\genzz\rbrack\rbrack\rbrack, \label{comm0inproof}\\
\Iso\left(\igenzzz\genzz^{-2}\genz^{h-1}\right) & = & \left\{
\begin{array}{ll}
       
      q^{h-1}(1-q)^{-h-1}\left(\ad\genzz\right)^{h-1}\left(\lbrack\lbrack\igenz,\igenzzz\rbrack,\igenzzz\rbrack\right), & h\in\Z^+, \\
     
      (-1)^{h-1}(1-q)^{h-3}\left(\ad\igenzz\right)^{1-h}\left(\lbrack\lbrack\igenz,\igenzzz\rbrack,\igenzzz\rbrack\right), & h\in\Z\backslash\Z^+.
\end{array}\right.\label{commTinproof}
\end{eqnarray}
We note that the scalar coefficients in \eqref{z2inL}, \eqref{comm0inproof}, \eqref{commTinproof} are all defined in the field because $q$ is assumed to be nonzero and not a root of unity. By inspecting the right-hand sides in \eqref{comm0inproof}, \eqref{commTinproof}, we find that both $\Iso\left(\genzzz\genzz^2\genz\right)$ and $\Iso\left(\igenzzz\genzz^{-2}\genz^{h-1}\right)$ are in $\TorLie$, and so by \eqref{z2inL}, we have $\genzz^h\in\TorLie$. A similar argument can be made to prove that $\genzzz^h\in\TorLie$, and this should start by applying $\Iso^2$ on both sides of \eqref{zinL}.\qed

\end{proof}

\noindent We now look for more basis elements from \eqref{TorBasis} that are in $\TorLie$. But first, we need the following.

\begin{proposition} For any $h,m,n\in\Z$, the relations
\begin{eqnarray}
\lbrack\genzzz^h,\genzz^m\rbrack & = & \left(1-q^{hm}\right)\genzzz^h\genzz^m,\label{comm3with2}\\
\lbrack\genzzz^h\genzz^m,\genz^n\rbrack & = & \left(1-q^{n(m-h)}\right)\genzzz^h\genzz^m\genz^n,\label{comm32with1}\\
\lbrack\genzzz^h,\genzz^h\genz^n\rbrack & = & \left(1-q^{h(h-n)}\right)\genzzz^h\genzz^m\genz^n,\label{comm3with21}
\end{eqnarray}
hold in $\TorAlg$.
\end{proposition}
\begin{proof} Use \eqref{structLie}.\qed
\end{proof}

\begin{lemma}\label{LieBasisLem} For any $h,m,n\in\Z\backslash\{0\}$, basis vectors of $\TorAlg$ from \eqref{TorBasis} of the form 
\begin{eqnarray}
\genzzz^h\genzz^m,\quad\quad\genzzz^h\genz^m,\quad\quad\genzz^h\genz^m, & \label{by2}\\
\genzzz^h\genzz^m\genz^n, & (h\neq m), \label{by3hm}\\
\genzzz^h\genzz^h\genz^n, & (h\neq n), \label{by3hn}
\end{eqnarray}
are elements of $\TorLie$.
\end{lemma}
\begin{proof} By Lemma~\ref{onepowerLem} and \eqref{comm3with2}, we find that $\genzzz^h\genzz^m\in\TorLie$. Rewrite \eqref{comm3with2} as
\begin{eqnarray}
\genzzz^m\genzz^h = (1-q^{hm})^{-1}\lbrack\genzzz^m,\genzz^h\rbrack.\label{twopowers}
\end{eqnarray}
Apply $\Iso$ to both sides of \eqref{twopowers}. Reorder the resulting left-hand side, and by making some adjustments in the scalar coefficients, we get 
\begin{eqnarray}
\genzzz^h\genz^m = q^{hm}(1-q^{hm})^{-1}\lbrack \genz^m,\genzzz^h\rbrack,\nonumber
\end{eqnarray}
where $\genz^m,\genzzz^h\in\TorLie$ by Lemma~\ref{onepowerLem}, and so we deduce that $\genzzz^h\genz^m\in\TorLie$. By a similar argument, we also have $\genzz^h\genz^m\in\TorLie$. Since it has now been established that $\genzzz^h\genzz^m$ and $\genz^n$ are both in $\TorLie$, by \eqref{comm32with1}, we find that the vectors \eqref{by3hm} are in $\TorLie$. The significance of the restriction $h\neq m$ is evident from the appearance of the scalar coefficient in \eqref{comm32with1}. By a similar argument, we find that the vectors \eqref{by3hn} are also in $\TorLie$.\qed
\end{proof}

\begin{lemma}\label{closedLem} If $h,m,n,u,v,w\in\Z$ such that $h+u=m+v=n+w$, then $\lbrack\genzzz^h\genzz^m\genz^n,\genzzz^u\genzz^v\genz^w\rbrack=0$.
\end{lemma}
\begin{proof} Set $u=-h$, $v=-m$, $w=-n$ in \eqref{structLie}.\qed
\end{proof}

\begin{lemma}\label{TorLieLem} A basis for the Lie algebra $\TorLie$ consists of the vectors
\begin{eqnarray}
\genzzz^h\genzz^m\genz^n,\label{TorLieBasis}
\end{eqnarray}
where at least one of the conditions $h\neq m$, $h\neq n$, or $m\neq n$ is true.
\end{lemma}
\begin{proof} Denote the span of the vectors in the statement by $\mathcal{K}$. Since such vectors are taken from the basis \eqref{TorBasis} of $\TorAlg$, these vectors are linearly independent, and so they form a basis for their span, which is $\mathcal{K}$. Thus, we are done if we show $\mathcal{K}=\TorLie$. We first prove that $\mathcal{K}$ is a Lie subalgebra of $\TorAlg$. What would suffice is to show that the Lie bracket of any two of the basis elements of $\mathcal{K}$ from the statement, say $\genzzz^h\genzz^m\genz^n$ and $\genzzz^u\genzz^v\genz^w$, is a linear combination of the same basis elements in the statement. In view of Remark~\ref{minorRem}\ref{Z3grad}, $\lbrack \genzzz^h\genzz^m\genz^n,\genzzz^u\genzz^v\genz^w\rbrack$ is in the $\Z^3$-gradation subspace $\trigradsub_{(h+u,m+v,n+w)}$, and is hence a scalar multiple of $\genzzz^{h+u}\genzz^{m+v}\genz^{n+w}$. But by considering the property of the basis elements of $\mathcal{K}$ in the statement, the only possibility for $\lbrack \genzzz^h\genzz^m\genz^n,\genzzz^u\genzz^v\genz^w\rbrack= c\genzzz^{h+u}\genzz^{m+v}\genz^{n+w}$ (for some scalar $c$) to be not in $\mathcal{K}$ is when $c\neq 0$ and $h+u=m+v=n+w$. But this is impossible by Lemma~\ref{closedLem}. Hence, $\lbrack \genzzz^h\genzz^m\genz^n,\genzzz^u\genzz^v\genz^w\rbrack\in\mathcal{K}$, and so $\mathcal{K}$ is a Lie subalgebra of $\TorAlg$. Observe that the basis vectors of $\mathcal{K}$ in the statement include all the generators of $\TorAlg$. Since the smallest Lie subalgebra of $\TorAlg$ that contains all the generators of $\TorAlg$ is $\TorLie$, we have $\TorLie\subseteq\mathcal{K}$. To get the other set inclusion, we simply make use of Lemmas~\ref{onepowerLem} and \ref{LieBasisLem}, which imply that all the basis vectors of $\mathcal{K}$ are in $\TorLie$. Therefore, $\mathcal{K}=\TorLie$.\qed
\end{proof}

\noindent By Lemma~\ref{TorLieLem} we are able to identify which vector subspace of $\TorAlg$ is precisely $\TorLie$. What remains of the vector space $\TorAlg$ can be easily described by the conditions on the exponents of the generators imposed on the basis elements of $\TorLie$ indicated in Lemma~\ref{TorLieLem}, and so we have the following.

\begin{corollary} A direct sum decomposition (of vector spaces) for $\TorAlg$ is given by
\begin{eqnarray}
\TorAlg=\TorLie\  \oplus\  \bigoplus_{h\in\Z}\trigradsub_{(h,h,h)}.\label{TorDirSum}
\end{eqnarray}
\end{corollary}

\noindent Define $\Proj$ as the vector space projection of $\TorAlg$ onto $\bigoplus_{h\in\Z}\trigradsub_{(h,h,h)}$. Equivalently, $\Proj$ is the canonical map $\TorAlg\rightarrow\TorAlg/\TorLie$ if $\TorAlg/\TorLie$ is viewed as a quotient \emph{of vector spaces}.

\begin{corollary}\label{LiesubCor} The Lie algebra $\FairOLie$ is a Lie subalgebra of $\TorLie$.
\end{corollary}
\begin{proof} For any $k\in\{1,2,3\}$, the generator $\genIarb_k$ of the Lie algebra $\FairOLie$ has the property $\Proj\left(\genIarb_k\right)=\Proj\left(\frac{\genGarb_k}{q-q^{-1}}\right)=0$ in view of the equations \eqref{normalG1}, \eqref{normalG2}, \eqref{normalG3}. Thus, $\FairOLie$ is a Lie algebra contained in $\TorLie$. \qed
\end{proof}

\section{Nonzero polynomials in $\Casimir$ are not Lie polynomials in $\genI$, $\genII$, $\genIII$}

Define the family $\TorGrad$ of vector subspaces of $\TorAlg$ by the property that for each $N\in\Z$, a basis for $\gradsub_N$ consists of the vectors 
\begin{eqnarray}
\genzzz^h\genzz^m\genz^n,\quad\quad\quad (h+m+n=N.)
\end{eqnarray}
We immediately find that $\TorGrad$ is a $\Z$-gradation of $\TorAlg$.
\begin{proposition} The elements $\genG$, $\genGG$, $\genGGG$ satisfy the properties
\begin{eqnarray}
\qhalf\genG^2+q^{-\frac{3}{2}}\genGG^2+\qhalf\genGGG^2 & \in & \bigoplus_{i=-2}^2\gradsub_{2i},\label{Gsquares}\\
-\genG\genGG\genGGG +q^{\frac{3}{2}}\genzzz^2\genzz^2\genz^2 & \in & \bigoplus_{i=-3}^2\gradsub_{2i}.\label{Gtriple}
\end{eqnarray}
\end{proposition}
\begin{proof} We first prove \eqref{Gsquares}. Let $H\in\{\genG,\genGG,\genGGG\}$. From \eqref{normalG1}, \eqref{normalG2}, \eqref{normalG3}, there exist 
\begin{eqnarray}
\lambda_{-2}\in\gradsub_{-2},\quad\quad\quad\quad\lambda_{0}\in\gradsub_{0},\quad\quad\quad\quad\lambda_{2}\in\gradsub_{2},\label{lambdadef}
\end{eqnarray} such that 
\begin{eqnarray}
H=\lambda_{-2}+\lambda_{0}+\lambda_{2}.\label{abstractG}
\end{eqnarray}
Consider the following elements of $\TorAlg$.
\begin{eqnarray}
\mu_{-4} & := & \lambda_{-2}^2,\label{lambda1}\\
\mu_{-2} & := & \lambda_{-2}\lambda_{0}+\lambda_{0}\lambda_{-2} ,\label{lambda2}\\
\mu_{0} & := & \lambda_{-2}\lambda_{2}+\lambda_{0}^2+\lambda_{2}\lambda_{-2} ,\label{lambda3}\\
\mu_{2} & := & \lambda_{0}\lambda_{2}+\lambda_{2}\lambda_{0} ,\label{lambda4}\\
\mu_{4} & := &\lambda_{2}^2.\label{lambda5}
\end{eqnarray}
By \eqref{lambdadef} and \eqref{lambda1} to \eqref{lambda5}, we find that $\mu_j\in\gradsub_j$ for any $j\in\{-4,-2,0,2,4\}$, and so  $\sum_{i=-2}^2\mu_{2i}\in\bigoplus_{i=-2}^2\gradsub_{2i}$. Also, it is routine to show that $H^2=\sum_{i=-2}^2\mu_{2i}$, and so we have
\begin{eqnarray}
H^2\in \bigoplus_{i=-2}^2\gradsub_{2i}.\label{whereisH2}
\end{eqnarray}
But since $H$ is arbitrary, \eqref{whereisH2} implies that any linear combination of $\genG^2,\genGG^2,\genGGG^2$ is an element of $\bigoplus_{i=-2}^2\gradsub_{2i}$. In particular, so is $\qhalf\genG^2+q^{-\frac{3}{2}}\genGG^2+\qhalf\genGGG^2$. This proves \eqref{Gsquares}. We now prove \eqref{Gtriple}. By \eqref{normalG1}, \eqref{normalG2}, \eqref{normalG3}, there exist 
\begin{eqnarray}
\alpha_{-2},\beta_{-2},\gamma_{-2}\in\gradsub_{-2},\quad\quad\quad\quad\alpha_{0},\beta_{0},\gamma_{0}\in\gradsub_{0},\label{alphadef}
\end{eqnarray} such that 
\begin{eqnarray}
\genG & = & \nqhalf\genzzz\genz+\alpha_{-2}+\alpha_{0},\label{Ghead1}\\
\genGG & = & \qhalf\genzzz\genzz+\beta_{-2}+\beta_{0},\label{Ghead2}\\
\genGGG & = & \qhalf\genzz\genz+\gamma_{-2}+\gamma_{0}.\label{Ghead3}
\end{eqnarray}
Similar to the technique of using \eqref{lambda1} to \eqref{lambda5} in the proof of \eqref{Gsquares}, we define the following elements of $\TorAlg$.
\begin{eqnarray}
\nu_{-4} & := & -\alpha_{-2}\beta_{-2},\label{nu1}\\
\nu_{-2} & := & -\alpha_{0}\beta_{-2}-\alpha_{-2}\beta_{0} ,\label{nu2}\\
\nu_{0} & := & -\nqhalf\genzzz\genz\beta_{-2}-\alpha_{0}\beta_{0}-\nqhalf\alpha_{-2}\genzzz\genzz,\label{nu3}\\
\nu_{2} & := & -\nqhalf\genzzz\genz\beta_{0}-\nqhalf\alpha_0\genzzz\genzz.\label{nu4}
\end{eqnarray}
\def\NuSum{\sum_{i=-2}^1\nu_{2i}}
\def\NuSumGroup{\left(\NuSum\right)}
It is routine to show that 
\begin{eqnarray}
-\genG\genGG & = & -\genzzz\genz\genzzz\genzz + \NuSum,\label{G1G2}\\
 & = & -\genzzz^2\genzz\genz + \NuSum,\label{G1G2ord}
\end{eqnarray}
where \eqref{G1G2ord} is obtained from \eqref{G1G2} by using the reordering formula from Section~\ref{ReorderSec} on the first term of \eqref{G1G2}. Now, we use \eqref{Ghead3} and \eqref{G1G2ord} to compute for $-\genG\genGG\genGGG$, and we obtain
\begin{eqnarray}
-\genG\genGG\genGGG & = & -\qhalf\genzzz^2(\genzz\genz)^2\label{threeGs1}\\
& & -\genzzz^2\genzz\genz\gamma_0 - \genzzz^2\genzz\genz\gamma_{-2}\label{threeGs2}\\
& & +\qhalf\NuSumGroup\genzz\genz\label{threeGs3}\\
& & +\NuSumGroup\gamma_0\label{threeGs4}\\
& & +\NuSumGroup\gamma_{-2}.\label{threeGs5}
\end{eqnarray}
By \eqref{alphadef} and \eqref{nu1} to \eqref{nu4}, we find that $\nu_j\in\gradsub_j$ for all $j\in\{-4,-2,0,2\}$. Then $\NuSum\in\bigoplus_{i=-2}^1\gradsub_{2i}$. This further implies that the terms \eqref{threeGs2} to \eqref{threeGs5} are elements of 
\begin{eqnarray}
\bigoplus_{i=1}^2\gradsub_{2i},\quad\quad\quad\bigoplus_{i=-2}^1\gradsub_{2i}\gradsub_{2},\quad\quad\quad\bigoplus_{i=-2}^1\gradsub_{2i},\quad\quad\quad\bigoplus_{i=-2}^1\gradsub_{2i}\gradsub_{-2},\label{insubs}
\end{eqnarray}
respectively. By the properties of the gradation subspaces, the second and fourth vector spaces in \eqref{insubs} can be simplified as 
\begin{eqnarray}
\bigoplus_{i=-2}^1\gradsub_{2i}\gradsub_{2}=\bigoplus_{i=-2}^1\gradsub_{2i+2}=\bigoplus_{i=-1}^2\gradsub_{2i},\quad\quad\quad\quad\bigoplus_{i=-2}^1\gradsub_{2i}\gradsub_{-2}=\bigoplus_{i=-2}^1\gradsub_{2i-2}=\bigoplus_{i=-3}^0\gradsub_{2i},\nonumber
\end{eqnarray}
and so we deduce that the terms \eqref{threeGs2} to \eqref{threeGs5} are elements of 
\begin{eqnarray}
\bigoplus_{i=1}^2\gradsub_{2i},\quad\quad\quad\bigoplus_{i=-1}^2\gradsub_{2i},\quad\quad\quad\bigoplus_{i=-2}^1\gradsub_{2i},\quad\quad\quad\bigoplus_{i=-3}^0\gradsub_{2i},\label{insubs2}
\end{eqnarray}
respectively. Denote by $R$ the sum of \eqref{threeGs2} to \eqref{threeGs5}. By inspection of the limits of the indices of the direct sums in \eqref{insubs2}, we find that $R\in \bigoplus_{i=-3}^2\gradsub_{2i}$. We now rewrite \eqref{threeGs1} to \eqref{threeGs5} as $-\genG\genGG\genGGG +\qhalf\genzzz^2(\genzz\genz)^2=R$, in which we further rewrite the term $\qhalf\genzzz^2(\genzz\genz)^2$ using the reordering formula for $\TorAlg$, and by doing this we get
\begin{eqnarray}
-\genG\genGG\genGGG +q^{\frac{3}{2}}\genzzz^2\genzz^2\genz^2 = R\quad\in\quad \bigoplus_{i=-3}^2\gradsub_{2i}.\qed\nonumber
\end{eqnarray}
\end{proof}

\begin{lemma}\label{mainLem} For any $n\in\Z^+$, the Casimir element $\Casimir$ satisfies the property
\begin{eqnarray}
\Casimir^n-(-1)^nq^{2(n^2-n+2)}(q^2-1)^{-2n}\genzzz^{2n}\genzz^{2n}\genz^{2n}\quad\in\quad\bigoplus_{i=-3n}^{3n-1}\gradsub_{2i}.\label{Cgrad}
\end{eqnarray}
\end{lemma}
\begin{proof} We use induction on $n$. By \eqref{Gsquares}, \eqref{Gtriple}, there exists $R\in\bigoplus_{i=-3}^2\gradsub_{2i}$ such that 
\begin{eqnarray}
-\genG\genGG\genGGG + \qhalf\genG^2+q^{-\frac{3}{2}}\genGG^2+\qhalf\genGGG^2+q^{\frac{3}{2}}\genzzz^2\genzz^2\genz^2 = R.\label{addGs}
\end{eqnarray}
Since we are considering $\FairO$ as the subalgebra of $\TorAlg$ described in Proposition~\ref{subalgProp}, we have $\genIarb_k=\frac{\genGarb_k}{q-q^{-1}}$ for any $k\in\{1,2,3\}$, and so we obtain from \eqref{CasimirDef} the relation
\begin{eqnarray}
q^{-\frac{5}{2}}(q^2-1)^2\Casimir = -\genG\genGG\genGGG + \qhalf\genG^2+q^{-\frac{3}{2}}\genGG^2+\qhalf\genGGG^2.\label{subGs}
\end{eqnarray}
Substituting using \eqref{subGs} into \eqref{addGs} and by some adjustments in  scalar coefficients such that $\Casimir$ has scalar coefficient $1$, we obtain
\begin{eqnarray}
\Casimir+q^4(q^2-1)^{-2}\genzzz^2\genzz^2\genz^2=q^\frac{5}{2}(q^2-1)^{-2}R\in\bigoplus_{i=-3}^2\gradsub_{2i},\label{nis1}
\end{eqnarray}
which is precisely \eqref{Cgrad} when $n=1$. Suppose \eqref{Cgrad} holds for some $n\in\Z^+$. Thus, for any $i\in\{-3n,-3n+1,\ldots,3n-1\}$ there exists $\nu_{2i}\in\gradsub_{2i}$ such that 
\begin{eqnarray}
\Casimir^n-(-1)^nq^{2(n^2-n+2)}(q^2-1)^{-2n}\genzzz^{2n}\genzz^{2n}\genz^{2n} = \sum_{i=-3n}^{3n-1}\nu_{2i}.\label{nisn}
\end{eqnarray}
Also, observe that in \eqref{nis1}, the expression  $q^\frac{5}{2}(q^2-1)^{-2}R$ is equal to $\sum_{i=-3n}^{3n-1}\mu_{2i}$ for some elements $\mu_{2i}$ of $\TorAlg$ with the property $\mu_{2i}\in\gradsub_{2i}$ for all $i\in\{-3n,-3n+1,\ldots,3n-1\}$. Thus, we can rewrite \eqref{nis1} and \eqref{nisn} as
\begin{eqnarray}
\Casimir & = & -q^4(q^2-1)^{-2}\genzzz^2\genzz^2\genz^2+\sum_{i=-3}^{2}\mu_{2i},\label{nis12}\\
\Casimir^n & = & (-1)^nq^{2(n^2-n+2)}(q^2-1)^{-2n}\genzzz^{2n}\genzz^{2n}\genz^{2n} + \sum_{i=-3n}^{3n-1}\nu_{2i}.\label{nisn2}
\end{eqnarray}
Use \eqref{nis12} and \eqref{nisn2} to solve for $C^{n+1}$. More precisely, one way is to multiply the left-hand sides as $C^n\cdot C$, and we describe in the following some characteristics of the resulting right-hand side, which can be verified by some routine calculations. First, there is one and only one term that is a scalar multiple of precisely the vector  $\genzzz^{2n}\genzz^{2n}\genz^{2n}\genzzz^2\genzz^2\genz^2$. It is routine to show that all other terms are elements of $\bigoplus_{i=-3(n+1)}^{3(n+1)-1}\gradsub_{2i}$. Use the reordering formula to get $\genzzz^{2(n+1)}\genzz^{2(n+1)}\genz^{2(n+1)}$ from $\genzzz^{2n}\genzz^{2n}\genz^{2n}\genzzz^2\genzz^2\genz^2$. The resulting scalar  coefficient is equal to $$(-1)^{(n+1)}q^{2((n+1)^2-(n+1)+2)}(q^2-1)^{-2(n+1)}.$$ By these observations, \eqref{Cgrad} holds for $n+1$. This completes the proof.\qed
\end{proof}

\begin{theorem} The sum of the vector spaces $\FairOLie$ and $Z\left(\FairO\right)$ is direct. i.e., Since $C$ generates $Z\left(\FairO\right)$, any nonzero polynomial in $\Casimir$ is not a Lie polynomial in $\genI,\genII,\genIII$.
\end{theorem}
\begin{proof} Let $U\in Z\left(\FairO\right)$. We have already established in \eqref{TorDirSum} that if $U=\algI\in\trigradsub_{(0,0,0)}$, then $U\notin\TorLie$, and hence by Corollary~\ref{LiesubCor}, we further have $U\notin\FairOLie$. Thus, without loss of generality, suppose that $U$ has polynomial degree $n\in\Z^+$, and that the coefficient of $\Casimir^n$ in $U$ is $1$. By Lemma~\ref{mainLem}, there exists a nonzero $L\in\trigradsub_{(2n,2n,2n)}$ and some $R\in\bigoplus_{i=-3n}^{3n-1}\gradsub_{2i}$ such that $U=L+R$. The condition $L\in\trigradsub_{(2n,2n,2n)}$ implies that
\begin{eqnarray}
\Proj(U)=L+\Proj(R).\label{projEq}
\end{eqnarray}
In view of the properties of the defining bases of the $\Z^3$-gradation subspaces $\trigradsub_{(h,m,n)}$ and those of the $\Z$-gradation subspaces $\gradsub_N$ of $\TorAlg$, we find that the conditions $L\in\trigradsub_{(2n,2n,2n)}$ and $R\in\bigoplus_{i=-3n}^{3n-1}\gradsub_{2i}$ imply  $R\notin\trigradsub_{(2n,2n,2n)}$, and so in the right-hand side of \eqref{projEq}, no summand in the linear combination for $\Proj(R)$ serves as the additive inverse of $L$, which further implies that $\Proj(U)\neq 0$. Thus, $U\notin\TorLie$, and therefore  $U\notin\FairOLie$.\qed
\end{proof}

\section*{Acknowledgements}

This work was done while the author was a post-doctoral research guest in M\"alardalen University, V\"aster\aa s, Sweden. He then wants to extend his thanks to his host, Prof. Sergei Silvestrov for providing the exposure to related research, and the environment that made this kind of work possible. The author also wishes to thank Dr. Lars Hellstr\"om for valuable comments and suggestions. The approach in determining a basis for $\TorLie$ done here is something that the author realized because of the works \cite{Can17,Can18} that he did on $q$-deformed Heisenberg algebras. This work was partially supported by a grant from De La Salle University, Manila, and also by a grant from the Commission for Developing Countries of the International Mathematical Union (IMU-CDC).


\begin{thebibliography}{35}




\bibitem{Berg} Bergman, G. (1978). The diamond lemma for ring theory. {\it Adv. Math.} 29(2):178-218.

\bibitem{Can15} Cantuba, R. (2015). A Lie algebra related to the universal Askey-Wilson algebra. {\it Matimy\'{a}s Matematika.} 38(2):51-75.

\bibitem{Can17} Cantuba, R. (2017). Lie polynomials in $q$-deformed Heisenberg algebras. arXiv:1709.02612.

\bibitem{Can18} Cantuba, R. (2018). A $q$-deformed Heisenberg algebra as a normed space. arXiv:1805.02362.

\bibitem{Chek} Chekhov, L., Fock., V. (2000). Observables in 3D gravity and geodesic algebras. {\it Czech. J. Phys.} 50(11):1201-1208.

\bibitem{Fair} Fairlie, D. (1990). Quantum deformations of $SU(2)$. {\it J. Phys. A.} 23(5):L183-L187.

\bibitem{Hav99} Havl\'\i\v cek, M., Klimyk U., Po\v sta, S. (1999). Representations of the cyclically symmetric $q$-deformed algebra $so_q(3)$. {\it J. Math. Phys.} 40(4): 2135-2161.

\bibitem{Hav01} Havl\'\i\v cek, M., Po\v sta, S. (2001). On the classification of irreducible finite-dimensional representations of $\FairO$. {\it J. Math. Phys.} 42(1): 472-500.

\bibitem{Hav11} Havl\'\i\v cek, M., Po\v sta, S. (2011). Center of quantum algebra $\FairO$. {\it J. Math. Phys.} 52(4): Article 043521.

\bibitem{Hel00} Hellstr\"{o}m, L., Silvestrov, S. (2000). {\it Commuting elements in $q$-deformed Heisenberg algebras.} Singapore: World Scientific.

\bibitem{Hel05} Hellstr\"{o}m, L., Silvestrov, S. (2005). Two-sided ideals in $q$-deformed Heisenberg algebras. {\it Expo. Math.} 23(2):99-125.

\bibitem{Iorg} Iorgov, N. (2002). On the center of the $q$-deformed algebra $\FairO$ related to quantum gravity at $q$ a root of $1$. {\it Proceedings of Institute of Mathematics of NAS of Ukraine.} 43(2):449-455.

\bibitem{Ito} Ito, T., Terwilliger, P., Weng C. (2006). The quantum algebra $\Equit$ and its equitable presentation. {\it J. Algebra.} 298(1):284-301. 

\bibitem{Odess} Odesskii, A. (1986). An analog of the Sklyanin algebra. {\it Func. Anal. Appl.} 20(2):152-153.













\bibitem{Ter}  Terwilliger, P. (2011). The universal Askey-Wilson algebra. {\it SIGMA.} 7(069): arXiv:1104.2813.


\bibitem{Zhed}  Zhedanov, A. (1991). ``Hidden symmetry" of the Askey-Wilson polynomials.  {\it Theoret. and Math. Phys.} 89(2):1146-1157.

\end{thebibliography}
\end{document}